\documentclass[11pt,a4paper]{article}
\usepackage[titletoc,title]{appendix}
\usepackage{amsmath,amsfonts,amssymb,bm,hyperref,amsthm,tikz}

\numberwithin{equation}{section}
\newtheorem{theorem}{Theorem}[section]
\newtheorem{proposition}[theorem]{Proposition}
\newtheorem{lemma}[theorem]{Lemma}

\theoremstyle{definition}

\newtheorem{remark}[theorem]{Remark}
\newtheorem{definition}[theorem]{Definition}

\usepackage[top=1in, bottom=1in, left=1in, right=1in]{geometry}
\newcommand{\Ad}{\operatorname{Ad}}
\newcommand{\Maps}{\operatorname{Maps}}
\renewcommand{\Im}{\operatorname{Im}}
\newcommand{\RP}{{\mathbb R}{\rm P}}

\begin{document}

\title{Analytic definition of spin structure}
\author{
Zhirayr Avetisyan
\thanks{ZA:
Department of Mathematics,
University College London,
Gower Street,
London WC1E~6BT,
UK;
z.avetisyan@ucl.ac.uk;
ZA was supported by EPSRC grant EP/M000079/1.
}
\and
Yan-Long Fang\thanks{YF:
Department of Mathematics,
University College London,
Gower Street,
London WC1E~6BT,
UK;
yan.fang.12@ucl.ac.uk.
}
\and
Nikolai Saveliev\thanks{NS:
Department of Mathematics,
University of Miami,
PO Box 249085
Coral Gables,
FL 33124,
USA;
saveliev@math.miami.edu,
\url{http://www.math.miami.edu/\~saveliev/};
NS was supported by
LMS grant 21420,  EPSRC grant EP/M000079/1, and the Simons Collaboration Grant 426269.
}
\and
Dmitri Vassiliev\thanks{DV:
Department of Mathematics,
University College London,
Gower Street,
London WC1E~6BT,
UK;
d.vassiliev@ucl.ac.uk,
\url{http://www.homepages.ucl.ac.uk/\~ucahdva/};
DV was supported by EPSRC grant EP/M000079/1.
}}

\renewcommand\footnotemark{}



\maketitle
\begin{abstract}
We work on a parallelizable time-orientable Lorentzian 4-manifold and prove that
in this case the notion of spin structure can be equivalently defined
in a purely analytic fashion.
Our analytic definition relies on the use of the concept of a
non-degenerate two-by-two formally self-adjoint first order linear differential operator
and gauge transformations of such operators.
We also give an analytic definition
of spin structure for the 3-dimensional Riemannian case.

\


\

{\bf MSC classes: }
primary 35R01; secondary 35L40, 35Q41, 53C27, 53C50.

\end{abstract}

\newpage

\tableofcontents

\section{Playing field}
\label{Playing field}

Let $M$ be a connected smooth 4-manifold without boundary, not necessarily compact.

In this paper we will be working with functions on $M$,
densities on $M$, vector fields on~$M$, metrics on $M$ etc, and all of these
will be assumed to be infinitely smooth.
We will also be working with differential operators acting on $M$
and the coefficients of these differential operators
will be assumed to be infinitely smooth.

We will use Latin letters for \emph{anholonomic (frame)} indices
and Greek letters for \emph{holonomic (tensor)} indices.
We will use the convention of summation
over repeated indices; this will apply both to
Greek and to Latin indices.

A half-density is a quantity $M\to\mathbb{C}$ which under changes of local
coordinates transforms as the square root of a density.
We will be working with compactly supported two-columns of half-densities
and will define the inner product on pairs $v,w$ of such 2-columns as
$\langle v,w\rangle:=\int_{M} w^*v\,dx\,$.
Here $x=(x^1,x^2,x^3,x^4)$ are local coordinates on $M$,
$dx=dx^1dx^2dx^3dx^4$ and the star stands for Hermitian conjugation.

Let $L$ be a first order linear differential
operator acting on 2-columns of half-densities.
In local coordinates this operator reads
$
L=F^\alpha(x)\frac\partial{\partial x^\alpha}+G(x)$,
where $F^\alpha(x)$ and $G(x)$ are some $2\times 2$ matrix-functions.
The problem here is that these matrix-functions are not invariant
under changes of local coordinates.
The standard way of providing an invariant analytic description of the
operator $L$ is by means of its principal and subprincipal symbols defined as
\begin{equation}
\label{definition of the principal symbol}
L_\mathrm{prin}(x,p):=iF^\alpha(x)\,p_\alpha\,,
\end{equation}
\begin{equation}
\label{definition of the subprincipal symbol}
L_\mathrm{sub}(x)
:=
G(x)+\frac i2(L_\mathrm{prin})_{x^\alpha p_\alpha}(x)
=
G(x)-\frac12(F^\alpha)_{x^\alpha}(x)
\end{equation}
respectively.
Here $p=(p_1,p_2,p_3,p_4)$ is the dual variable (momentum)
and the subscripts indicate partial derivatives.
It is known that $L_\mathrm{prin}$ and $L_\mathrm{sub}$
are invariantly defined matrix-functions
on $T^*M$ and $M$ respectively,
see
\cite[subsection 2.1.3]{mybook}
and
\cite[Appendix A]{action}
for details.
Furthermore,
it is easy to see that the principal and sub\-principal symbols
uniquely determine our first order differential operator
and that our operator is formally self-adjoint
if and only if its principal and subprincipal symbols are Hermitian.

Further on we work with $2\times2$ formally self-adjoint first order differential operators.

\begin{definition}
\label{definition of non-degeneracy}
We say that the operator $L$ is \emph{non-degenerate} if
\begin{equation}
\label{definition of non-degeneracy equation}
L_\mathrm{prin}(x,p)\ne0,\qquad\forall(x,p)\in T^*M\setminus\{0\}.
\end{equation}
\end{definition}

Not every 4-manifold admits a non-degenerate operator. The following lemma
establishes the appropriate criterion.

\begin{lemma}
\label{lemma on parallelizability}
The manifold $M$ admits a non-degenerate operator $L$ if and only if it is parallelizable.
\end{lemma}

\begin{proof}
Decomposing $L_\mathrm{prin}(x,p)$ with respect to the standard basis
\begin{equation}
\label{standard basis}
s^1=
\begin{pmatrix}
0&1\\
1&0
\end{pmatrix},
\quad
s^2=
\begin{pmatrix}
0&-i\\
i&0
\end{pmatrix},
\quad
s^3=
\begin{pmatrix}
1&0\\
0&-1
\end{pmatrix},
\quad
s^4=
\begin{pmatrix}
1&0\\
0&1
\end{pmatrix}
\end{equation}
in the real vector space of $2\times2$ Hermitian matrices, we get
\begin{equation}
\label{principal symbol via frame 2}
L_\mathrm{prin}(x,p)=s^j e_j{}^\alpha(x)\,p_\alpha\,,
\end{equation}
where the $e_j$, $j=1,2,3,4$, are some real-valued vector fields.
Here each vector $e_j(x)$ has coordinate components
$e_j{}^\alpha(x)$, $\alpha=1,2,3,4$.

Formula \eqref{principal symbol via frame 2} establishes a one-to-one correspondence
between principal symbols and tetrads\footnote{Further on,
in Section~\ref{Main result},
we will start referring to this tetrad of vector fields as a \emph{framing}.
This will be a purely terminological change,
for the benefit of readers accustomed to terminology used in topology and differential geometry.}
of vector fields.
Furthermore, formula \eqref{principal symbol via frame 2} allows us to
rewrite the non-degeneracy condition \eqref{definition of non-degeneracy equation} as
\begin{equation}
\label{definition of non-degeneracy equation in terms of vector fields}
\det e_j{}^\alpha(x)\ne0,\qquad\forall x\in M.
\end{equation}
But condition \eqref{definition of non-degeneracy equation in terms of vector fields}
is the condition of linear independence of the vector fields $e_j\,$.
\end{proof}

\begin{remark}
The critical element of the above poof is the fact that the dimension
of our manifold equals the dimension of the real vector space
of $2\times2$ Hermitian matrices.
\end{remark}

Further on we assume that our 4-manifold $M$ is parallelizable.

\section{Correspondence between operators and Lorentzian metrics}
\label{Correspondence between operators and Lorentzian metrics}

Observe that the determinant of the principal symbol is a quadratic form
in momentum $p$\,:
\begin{equation}
\label{definition of metric}
\det L_\mathrm{prin}(x,p)=-g^{\alpha\beta}(x)\,p_\alpha p_\beta\,.
\end{equation}
We interpret the real coefficients $g^{\alpha\beta}(x)=g^{\beta\alpha}(x)$,
$\alpha,\beta=1,2,3,4$, appearing in formula (\ref{definition of metric})
as components of a (contravariant) metric tensor.
It is known \cite[Lemma 2.1]{nongeometric} that this metric is
Lorentzian, i.e.~it has three positive eigenvalues and one negative
eigenvalue.

It turns out that the Lorentzian metric defined in accordance with formula
\eqref{definition of metric} has an additional geometric property.
In order to describe this property we need some definitions.

\begin{definition}
\label{definition of timelike}
A vector field $u$ is said to be
\emph{spacelike} if $(g_{\alpha\beta}\,u^\alpha u^\beta)(x)>0$, $\forall x\in M$,
\emph{lightlike} if $(g_{\alpha\beta}\,u^\alpha u^\beta)(x)=0$, $\forall x\in M$,
and
\emph{timelike} if $(g_{\alpha\beta}\,u^\alpha u^\beta)(x)<0$, $\forall x\in M$.
\end{definition}

\begin{definition}
\label{definition of time-orientable}
The Lorentzian manifold $(M,g)$ is said to be \emph{time-orientable}
if it admits a timelike vector field.
\end{definition}

Observe that the trace of the principal symbol is a linear form
in momentum $p$ and the coefficients of this linear form
can be interpreted as components of a vector field.
With the standard choice of Pauli matrices \eqref{standard basis}
the linear form in question reads
\begin{equation}
\label{trace of principal symbol}
\operatorname{tr}L_\mathrm{prin}(x,p)=2e_4{}^\alpha p_\alpha\,.
\end{equation}
It is easy to see that the vector field $e_4$ is timelike.
Thus, formula \eqref{definition of metric} defines a time-orientable
Lorentzian metric on our parallelizable
4-manifold $M$.

Let us now perform the above argument the other way round.
Suppose we have a parallelizable 4-manifold $M$ equipped with
a time-orientable Lorentzian metric.
In this case we have a timelike vector field which we will denote by $e_4$.
Without loss of generality we may assume that this timelike vector field is normalised,
i.e.~that $(g_{\alpha\beta}\,e_4{}^\alpha e_4{}^\beta)(x)=-1$, $\forall x\in M$.

\begin{lemma}
\label{lemma 2 point 3}
One can choose vector fields $e_1$, $e_2$ and $e_3$ such that the tetrad $e_j$, $j=1,2,3,4$, 
is linearly independent at all points of the manifold.
\end{lemma}

\begin{proof}
Let us fix a trivialization $TM = M \times \mathbb R^4$ and view $e_4$ as a smooth map $M \to S^3$. Since $S^3$ is parallelizable, its orthonormal frame bundle $SO(3) \to SO(4) \to S^3$ is trivial and therefore admits a section $S^3 \to SO(4)$. Composing this section with the map $M \to S^3$ we obtain the desired tetrad. 
\end{proof}

\begin{remark}
\label{2 and 8}
The above argument also works for parallelizable manifolds of dimension 2 and~8.
\end{remark}

Applying now the Gram--Schmidt process
we obtain new vector fields $e_1$, $e_2$, $e_3$ and $e_4$ such that
\begin{equation}
\label{Lorentzian orthonormality}
(g_{\alpha\beta}\,e_j{}^\alpha e_k{}^\beta)(x)
=
\begin{cases}
\phantom{+}0\quad\text{if}\quad j\ne k,\\
+1\quad\text{if}\quad j=k\ne4,
\\
-1\quad\text{if}\quad j=k=4,
\end{cases}
\end{equation}
for all $x\in M$.
Here the Gram--Schmidt process works because the restriction of a Lorentzian metric
to the orthogonal complement of a timelike vector gives a Riemannian metric.
Finally, substituting our tetrad $e_j$, $j=1,2,3,4$, into \eqref{principal symbol via frame 2}
we obtain a principal symbol with the property~\eqref{definition of metric}.

\section{Gauge transformations and spin structure}
\label{Gauge transformations and spin structure}

From now on the time-orientable Lorentzian metric is assumed to be fixed.
We will work with all possible
$2\times2$ formally self-adjoint non-degenerate first order linear differential operators
corresponding,
in the sense of formula \eqref{definition of metric},
to the given metric.
It was shown in the previous section that the set of such operators is non-empty.
Our aim is to classify operators corresponding to the given metric.

We specify an orientation on our manifold
and define the topological charge of our operator~as
\begin{equation}
\label{definition of relative orientation}
\mathbf{c}_\mathrm{top}:=
-\frac i2\sqrt{|\det g_{\alpha\beta}|}\,\operatorname{tr}
\bigl(
(L_\mathrm{prin})_{p_1}
(L_\mathrm{prin})_{p_2}
(L_\mathrm{prin})_{p_3}
(L_\mathrm{prin})_{p_4}
\bigr)=\operatorname{sgn}\det e_j{}^\alpha,
\end{equation}
with the subscripts $p_1$, $p_2$, $p_3$, $p_4$
indicating partial derivatives with respect to the components of momentum
$p=(p_1,p_2,p_3,p_4)$.
It is easy to see that
the number $\mathbf{c}_\mathrm{top}$ defined by formula
\eqref{definition of relative orientation}
can take only two values, $+1$ or $-1$,
and describes the orientation of the principal symbol
relative to our chosen orientation of local coordinates $x=(x^1,x^2,x^3,x^4)$.

Let us choose a timelike covector field $q$ and use it as a reference.
It is easy to see that the real-valued scalar function
$\operatorname{tr} L_\mathrm{prin}(x,q(x))$ does not vanish.
We define the temporal charge of our operator as
\begin{equation}
\label{definition of temporal charge}
\mathbf{c}_\mathrm{tem}:=
\operatorname{sgn}\operatorname{tr}
L_\mathrm{prin}(x,q(x))
=\operatorname{sgn}(q_\alpha e_4{}^\alpha)\,,
\end{equation}
see also formula \eqref{trace of principal symbol}.
Note that $q_\alpha e_4{}^\alpha\ne0$ because both $q$ and $e_4$ are timelike.
The number $\mathbf{c}_\mathrm{tem}$ defined by formula
\eqref{definition of temporal charge}
describes the orientation of the principal symbol
relative to our chosen time orientation.

We perform a primary classification of our operators based on the values of their
topological and temporal charges. Obviously,
the four classes of operators in this primary classification correspond to the
four connected components of the Lorentz group.

Further on we assume the topological and temporal charges to be fixed.


In order to classify our operators further we
introduce an arbitrary smooth $2\times2$ complex-valued matrix-function $R$
of determinant one,
\begin{equation}
\label{matrix-function R}
R:M\to SL(2,\mathbb{C}),
\end{equation}
and consider the transformation
\begin{equation}
\label{gauge transformation equation}
L\mapsto R^*LR.
\end{equation}
The transformation
\eqref{gauge transformation equation}
of our differential operator $L$ induces the transformation
\begin{equation}
\label{SL2C transformation of the principal symbol}
L_\mathrm{prin}\mapsto R^*L_\mathrm{prin}R
\end{equation}
of its principal symbol.
(The induced transformation of the subprincipal symbol will be considered
later in Section \ref{Classification beyond spin structure}.)
Formulae
\eqref{definition of metric},
\eqref{matrix-function R}
and
\eqref{SL2C transformation of the principal symbol}
imply that the transformation
\eqref{gauge transformation equation}
preserves our Lorentzian metric $g$.
Furthermore, it is easy to see
that the transformation
\eqref{gauge transformation equation}
preserves the topological charge
\eqref{definition of relative orientation}
and the temporal charge
\eqref{definition of temporal charge}.

We interpret \eqref{gauge transformation equation} as a gauge transformation
because it preserves the prescribed metric as well as the prescribed
topological and temporal charges.

Our analytic definition of spin structure is formulated as follows.

\begin{definition}
\label{definition of equivalent operators}
We say that the operators $L$ and $\tilde L$ are \emph{equivalent} if
\begin{equation}
\label{equivalence of principal symbols}
\tilde L_\mathrm{prin}=R^*L_\mathrm{prin}R
\end{equation}
for some smooth matrix-function \eqref{matrix-function R}.
An equivalence class of operators is called \emph{spin structure}.
\end{definition}

\section{Main result}
\label{Main result}

In this section we will show that our analytic definition of spin structure, Definition~\ref{definition of equivalent operators}, is equivalent to the traditional geometric definition. 

We begin by restating our analytic definition
Definition~\ref{definition of equivalent operators} in terms of framings.  By a \emph{frame} at a point $x \in M$ we mean a positively oriented and positively time-oriented orthonormal, in the Lorentzian sense \eqref{Lorentzian orthonormality},
frame in the tangent space $T_x M$, and by a \emph{framing} of $M$ a choice of a frame at every point $x \in M$ depending smoothly on the
point\footnote{
The terminology used in mathematical literature
and theoretical physics literature
is somewhat different.
In mathematical literature a frame refers to a set of vectors at a given point,
whereas in theoretical physics literature a frame refers to a set of vector fields.
In the current section as well as
in Section \ref{The 3-dimensional Riemannian case}
we use terminology prevalent in mathematical literature}.
In our case we have an explicit formula
\eqref{principal symbol via frame 2}
establishing a one-to-one correspondence
between principal symbols and framings.
Any two framings of the same manifold $M$ are related by a uniquely defined smooth function $f: M \to SO^+(3,1)$, where $SO^+(3,1)$ is the identity component of the Lorentz group.
Rephrasing Definition~\ref{definition of equivalent operators},
we will say that two framings are \emph{equivalent} if the function $f$ relating them factors as 
\[
f: M \rightarrow SL(2,\mathbb C) \overset{\Ad}{\longrightarrow} SO^+(3,1),
\]
where $\Ad: SL(2,\mathbb C) \to SO^+(3,1)$ is the adjoint representation. A spin structure on $M$ is then an equivalence class of framings. 

Note that a choice of reference framing on $M$ provides a trivialization of the tangent bundle $TM$ so that any other framing is related to this reference framing by a smooth function $f: M \to SO^+(3,1)$. Two framings corresponding to functions $f_1$ and $f_2$ are equivalent in the above sense if and only if there exists a smooth function $h: M \to SL(2,\mathbb C)$ such that $f_2 \cdot \Ad h = f_1$ as functions $M \to SO^+(3,1)$.

As the traditional definition of Lorentzian
spin structure we will use the definition from Baum\cite{baum_book,baum},
see also Bichteler \cite{bichteler}.
In the special case at hand, using the trivialization of the tangent bundle $TM$ via the reference frame, it reads as follows. A spin structure on $M$ is an equivalence class of commutative diagrams
\smallskip
\begin {equation*}
\begin{tikzpicture}
\draw (8,6) node (a) {$M \times SL(2,\mathbb C)$};
\draw (8,3) node (b) {$M \times SO^+(3,1)$};
\draw (11,4.5) node (c) {$M$};
\draw[->](a)--(b) node [midway,right](TextNode){$\Phi$};
\draw[->](a)--(c) node [midway,above](TextNode){$\pi$};
\draw[->](b)--(c) node [midway,above](TextNode){$\pi$};
\end{tikzpicture}
\end {equation*}
where $\pi$ stands for the projection onto the first factor, and the map $\Phi$ is equivariant in that $\Phi(x,g) = \Phi (x,1)\cdot \Ad g$ for all $x \in M$ and $g \in SL(2,\mathbb C)$. Two diagrams as above with the vertical maps $\Phi_1$ and $\Phi_2$ are called equivalent if there is a commutative diagram
\begin {equation*}
\begin{tikzpicture}
\draw (8,6) node (a) {$M \times SL(2,\mathbb C)$};
\draw (12,6) node (b) {$M \times SL(2,\mathbb C)$};
\draw (10,3.5) node (c) {$M \times SO^+(3,1)$};
\draw[->](a)--(b) node [midway,above](TextNode){$A$};
\draw[->](a)--(c) node [midway,left](TextNode){$\Phi_1\;$};
\draw[->](b)--(c) node [midway,right](TextNode){$\;\;\Phi_2$};
\end{tikzpicture}
\end {equation*}
such that $\pi\circ A = \pi$ and the map $A$ is equivariant in that $A(x,g) = A(x,1)\cdot g$ for all $x \in M$ and $g \in SL(2,\mathbb C)$.

\begin{theorem}
\label{Theorem Lorentzian}
For parallelizable time-orientable Lorentzian 4-manifolds the two definitions of spin structure,
our analytic definition and the traditional one,
are equivalent.
\end{theorem}

\begin{proof}
Using the commutativity of the first diagram, write $\Phi (x,g) = (x,\phi(x,g))$ for some function $\phi: M \times SL(2,\mathbb C) \to SO^+(3,1)$ and observe that the equivariance condition on $\Phi$ translates into the equation $\phi(x,g) = \phi(x,1)\cdot \Ad g$. Therefore, the map $\Phi$ is uniquely determined by the map $\psi: M \to SO^+(3,1)$ given by $\psi(x) = \phi(x,1)$. 

Similarly, write $A(x,g) = (x,\alpha(x,g))$ and observe that the equivariance condition on $A$ translates into the equation $\alpha(x,g) = \alpha(x,1)\cdot g$. Therefore, the map $A$ is uniquely determined by the map $\beta: M \to SL(2,\mathbb C)$ given by $\beta(x) = \alpha(x,1)$. One can easily check that the second commutative diagram then simply means that $\psi_2\cdot\Ad\beta = \psi_1$ as functions $M \to SO^+(3,1)$.

This completes the proof of the equivalence of two definitions of spin structure. Note that the equivalence we established is not canonical in that it depends on the choice of reference frame.
\end{proof}

\begin{remark}
According to \cite[Theorem 2]{baum}, the equivalence classes of Lorentzian spin structures on $M$ are classified by the cohomology group $H^1(M; \mathbb Z_2)$. This is an analogue of the well-known classification of Riemannian spin structures, see Proposition \ref{7 point 4} and Remark \ref{7 point 4 remark}.
\end{remark}


\section{Topological and geometric restrictions}
\label{Topological and geometric restrictions}

In Sections \ref{Playing field}--\ref{Gauge transformations and spin structure}
we gave an analytic definition of the concept of Lorentzian spin structure. This definition works
only for parallelizable 4-manifolds equipped with time-orientable Lorentzian metric.
In Section \ref{Main result} we proved, under the assumptions of parallelizability
and time-orientability, equivalence of our analytic definition to the traditional one.
It is natural to examine how restrictive these assumptions are.

\begin{proposition}
A non-compact time-orientable Lorentzian 4-manifold $M$ is parallelizable if and only if it is spin. 
\end{proposition}

\begin{proof}
According to Theorem 1 of Baum \cite{baum}, a time-orientable Lorentzian manifold $M$ is spin if and only if $w_2 (M) = 0$. Therefore, if $M$ is parallelizable it is obviously spin. Conversely, the tangent bundle of $M$ is classified by a homotopy class of maps $M \to BSO^+(3,1)$. Since the inclusion of $SO(3)$ as a subgroup into $SO^+(3,1)$ is a deformation retract, the classifying spaces $BSO^+(3,1)$ and $BSO(3)$ are homotopy equivalent. According to Dold--Whitney \cite{dold-whitney}, the homotopy classes of maps $M \to BSO(3)$ are classified by the second Stiefel--Whitney class $w_2(M)$ and the first Pontryagin class $p_1(M) \in H^4 (M;\mathbb Z)$. Since $M$ is non-compact, the group $H^4(M;\mathbb Z)$ vanishes, and the result follows. 
\end{proof}

In theoretical physics the prevalent view is that physically meaningful spacetimes
(Lorentzian 4-manifolds) are those that are non-compact and time-orientable.
Thus, for physically meaningful spacetimes our parallelizability assumption is
the necessary and sufficient condition for the existence of spin structure.

One can, of course, adopt a purely mathematical approach and study Lorentzian 4-manifolds
that are either compact or non-time-orientable or both. For such Lorentzian 4-manifolds
our analytic definition of spin structure may not work.

\section{Classification beyond spin structure}
\label{Classification beyond spin structure}

Let us consider all possible operators corresponding to the specified Lorentzian metric
(see \eqref{definition of metric}),
with specified charges
(see \eqref{definition of relative orientation}
and \eqref{definition of temporal charge})
and specified spin structure (see Definition~\ref{definition of equivalent operators}).
It turns out that it is possible to classify them further as follows.

Let us define the \emph{covariant subprincipal symbol}
 $\,L_\mathrm{csub}(x)\,$
in accordance with formula
\begin{equation}
\label{definition of covariant subprincipal symbol}
L_\mathrm{csub}:=
L_\mathrm{sub}
+\frac i{16}\,
g_{\alpha\beta}
\{
L_\mathrm{prin}
,
\operatorname{adj}L_\mathrm{prin}
,
L_\mathrm{prin}
\}_{p_\alpha p_\beta},
\end{equation}
where
$
\{F,G,H\}:=F_{x^\alpha}GH_{p_\alpha}-F_{p_\alpha}GH_{x^\alpha}
$
is the generalised Poisson bracket on matrix-functions
and $\,\operatorname{adj}\,$ is the operator of matrix adjugation
\begin{equation}
\label{definition of adjugation}
F=\begin{pmatrix}a&b\\ c&d\end{pmatrix}
\mapsto
\begin{pmatrix}d&-b\\-c&a\end{pmatrix}
=:\operatorname{adj}F
\end{equation}
from elementary linear algebra.
Now take an arbitrary smooth matrix-function \eqref{matrix-function R}
and consider the gauge transformation \eqref{gauge transformation equation}.
It was shown in \cite{nongeometric}
that the transformation
\eqref{gauge transformation equation}
of our differential operator $L$ induces the transformation
\begin{equation}
\label{SL2C transformation of the covariant subprincipal symbol}
L_\mathrm{csub}\mapsto R^*L_\mathrm{csub}R
\end{equation}
of its covariant subprincipal symbol.

Comparing formulae
\eqref{definition of the subprincipal symbol}
and
\eqref{definition of covariant subprincipal symbol}
we see that the standard subprincipal symbol
and covariant subprincipal symbol have the same structure, only
the covariant subprincipal symbol has a second correction term
designed to `take care of' special linear transformations
in the vector space of unknowns $v:M\to\mathbb{C}^2$.
The standard subprincipal symbol \eqref{definition of the subprincipal symbol}
is invariant under changes of local coordinates
(its elements behave as scalars),
whereas the covariant subprincipal
symbol~\eqref{definition of covariant subprincipal symbol}
retains this feature but gains an extra $SL(2,\mathbb{C})$
covariance property. In other words, the covariant subprincipal symbol
\eqref{definition of covariant subprincipal symbol}
behaves `nicely' under a wider group of transformations.

Formulae
\eqref{SL2C transformation of the principal symbol}
and
\eqref{SL2C transformation of the covariant subprincipal symbol}
imply that the covariant subprincipal
symbol can be uniquely represented in the form
\begin{equation}
\label{decomposition of  covariant subprincipal symbol}
L_\mathrm{csub}(x)=L_\mathrm{prin}(x,A(x)),
\end{equation}
where $A=(A_1,A_2,A_3,A_4)$ is some real-valued covector field.
We interpret the covector field appearing in formula
\eqref{decomposition of  covariant subprincipal symbol}
as the electromagnetic covector potential.

It is easy to see that the electromagnetic covector potential
is invariant under gauge transformations
\eqref{gauge transformation equation},
so it can be used for the purpose of further classification of our operators:
the electromagnetic covector potential
defines the operator uniquely modulo a transformation
\eqref{gauge transformation equation}.
Note, however, that this finer classification
is not particularly interesting from the topological
perspective because covector fields form a vector space.

\section{The 3-dimensional Riemannian case}
\label{The 3-dimensional Riemannian case}


Let $M$ be a connected smooth 3-manifold without boundary, not necessarily compact.
As in Section \eqref{Playing field},
let $L$ be a $2\times2$ formally self-adjoint non-degenerate
first order linear differential operator.
In dealing with the 3-dimensional case we make the additional assumption
\begin{equation}
\label{principal symbol is trace-free}
\operatorname{tr}L_\mathrm{prin}(x,p)=0,\qquad\forall(x,p)\in T^*M.
\end{equation}
Note that imposing condition \eqref{principal symbol is trace-free}
in the 4-dimensional setting would not make sense because it would
contradict non-degeneracy \eqref{definition of non-degeneracy equation}.

It is easy to see that under the assumption
\eqref{principal symbol is trace-free}
the non-degeneracy condition
\eqref{definition of non-degeneracy equation}
for our $2\times2$ operator $L$
is equivalent to the condition
\begin{equation}
\label{definition of ellipticity}
\det L_\mathrm{prin}(x,p)\ne0,\qquad\forall(x,p)\in T^*M\setminus\{0\}.
\end{equation}
But \eqref{definition of ellipticity} is the standard ellipticity condition.
Thus, in this section we work with
$2\times2$ formally self-adjoint elliptic
first order linear differential operators $L$ with trace-free principal symbols
which act over a connected smooth 3-manifold $M$ without boundary.

By analogy with Lemma \eqref{lemma on parallelizability} we have
\begin{lemma}
\label{lemma on parallelizability 3D}
The manifold $M$ admits an elliptic operator $L$
with trace-free principal symbol
if and only if it is parallelizable.
\end{lemma}

The proof of Lemma \ref{lemma on parallelizability 3D}
is similar to the proof of Lemma \ref{lemma on parallelizability}.

Further on we assume that our 3-manifold $M$ is parallelizable.

It is known \cite{Stiefel,Kirby}
that a 3-manifold is parallelizable if and only if it is orientable.
Therefore, further on we assume that our 3-manifold $M$ is orientable.
Orientability is our only topological restriction.

We define the metric in accordance with formula \eqref{definition of metric}.
It is easy to see that this metric is Riemannian.

From now on the Riemannian metric is assumed to be fixed.
We will work with all possible
$2\times2$ formally self-adjoint elliptic
first order linear differential operators with trace-free principal symbols
corresponding,
in the sense of formula \eqref{definition of metric},
to the given metric.
Arguing as in Section~\ref{Correspondence between operators and Lorentzian metrics},
it is easy to see that the set of such operators is non-empty.

We specify an orientation on our manifold
and define the topological charge of our operator~as
\begin{equation}
\label{definition of relative orientation 3D}
\mathbf{c}_\mathrm{top}:=
-\frac i2\sqrt{\det g_{\alpha\beta}}\,\operatorname{tr}
\bigl(
(L_\mathrm{prin})_{p_1}
(L_\mathrm{prin})_{p_2}
(L_\mathrm{prin})_{p_3}
\bigr)=\operatorname{sgn}\det e_j{}^\alpha,
\end{equation}
compare with formula \eqref{definition of relative orientation}.
Of course, as we are now working in the 3-dimensional setting,
the free indices in formula \eqref{definition of relative orientation 3D}
run through the values $1,2,3$.
Further on we assume the topological charge to be fixed.

In order to classify our operators further we
introduce an arbitrary smooth $2\times2$ special unitary matrix-function $R$,
\begin{equation}
\label{matrix-function R 3D}
R: M\to SU(2),
\end{equation}
and, as in Section \ref{Gauge transformations and spin structure},
consider the gauge transformation \eqref{gauge transformation equation}.
Comparing formulae \eqref{matrix-function R} and \eqref{matrix-function R 3D}
we see that we are now more restrictive in our choice of matrix-functions $R$,
which is because we want to preserve the condition \eqref{principal symbol is trace-free}.

We define spin structure in the 3-dimensional Riemannian setting in accordance with Definition~\ref{definition of equivalent operators}, having in mind the restricted choice of operators and gauge transformations.
We begin by restating our analytic definition
in terms of framings.
By a \emph{frame} at a point $x \in M$ we mean a positively oriented orthonormal frame in the tangent space $T_x M$, and by a \emph{framing} of $M$ a choice of a frame at every point $x \in M$ depending smoothly on the point.
Framings exist because all orientable 3-manifolds $M$ are parallelizable.
Furthermore, in our case we have an explicit formula
establishing a one-to-one correspondence
between trace-free principal symbols and framings:
this is formula
\eqref{principal symbol via frame 2}
with indices $j$ and $\alpha$ restricted to the set of values $1,2,3$.
Note that this one-to-one correspondence
between trace-free principal symbols and framings
was first observed in \cite[Appendix A]{cdv}.
Any two framings of the same manifold $M$ are related by a uniquely defined smooth function $f: M \to SO(3)$.
Rephrasing Definition~\ref{definition of equivalent operators}
with the restricted choice of operators and gauge transformations,
we will say that two framings are \emph{equivalent} if the function $f$ relating them factors as 
\[
f: M \to SU(2) \overset{\Ad}{\longrightarrow} SO(3),
\]
where $\Ad: SU(2) \to SO(3)$ is the adjoint representation. A spin structure on $M$ is then an equivalence class of framings.

The following result was announced, without proof, in \cite[Section 7]{jst_review}.

\begin{theorem}
\label{Theorem Riemannian}
For orientable Riemannian 3-manifolds
our analytic definition of spin structure
is equivalent to the standard definition of
\cite[Section~2.1]{friedrich}.
\end{theorem}

\begin{proof}
The proof of equivalence from Theorem \ref{Theorem Lorentzian}
goes through with little change once we replace the adjoint representation $\Ad: SL(2,\mathbb C) \to SO^+(3,1)$ by the adjoint representation $\Ad: SU(2) \to SO(3)$.
\end{proof}

In what follows, we provide a different proof of Theorem \ref{Theorem Riemannian}. This proof only works when $M$ is compact but it has the advantage that the equivalence it provides is canonical. The standard definition of spin structure we will use in this proof is the one often used by topologists; it can be found in Kaplan \cite{kaplan}, Definition 1.7, or Milnor \cite{milnor}, Alternative Definition 2 (see also Remark~\ref{R:milnor} regarding the latter definition at the end of this section). 

According to that definition, a spin structure on $M$ is a homotopy class of \emph{almost-framings} of $M$. By an almost-framing of $M$ one means a framing of the punctured manifold
\[
M_0:= M\setminus\{\text{point}\}.
\] 
Two framings of $M_0$ are said to be homotopic if they can be connected by a path of framings of $M_0$. Any framing of $M$ gives rise to a homotopy class of almost-framings of $M$ by restricting it to $M_0$ and taking the homotopy class of this restriction.

\begin{proposition}\label{P:one}
This gives a well-defined map $\psi$ from the set of spin-structures on $M$ as defined above to the set of the homotopy classes of almost-framings of $M$.
\end{proposition}

\begin{proof}
A function $f: M \to SO(3)$ relating two framings of $M$ restricts to a function $f_0: M_0 \to SO(3)$ relating their restrictions to $M_0$. If the function $f$ factors through $M \to SU(2)$, its restriction $f_0$ factors through $M_0 \to SU(2)$. However, every function $M_0 \to SU(2)$ is homotopic to a constant function, which in addition can be chosen to send the entire $M_0$ to the identity element of $SU(2)$. The composition of this homotopy with the adjoint representation $SU(2) \to SO(3)$ then provides a homotopy of the induced framings of $M_0$.
\end{proof}

\begin{proposition}\label{P:two}
The map $\psi$ is a bijection from the set of spin-structures on $M$ as defined above to the set of the homotopy classes of almost-framings of $M$.
\end{proposition}

\begin{proof}
For the purposes of this proof, fix a reference framing of $M$ so that any other framing is obtained from the reference framing by applying a function $f: M \to SO(3)$. This identifies the set of all framings of $M$ with the set $\Maps(M,SO(3))$. The equivalence relation on the framings translates into an equivalence relation on $\Maps(M,SO(3))$, two functions $f, g: M \to SO(3)$ being equivalent if and only if there exists a function $h: M \to SU(2)$ such that $g(x) = \Ad h(x) \cdot f(x)$ for all $x \in M$. Note that the point-wise multiplication makes both $\Maps(M,SU(2))$ and $\Maps(M,SO(3))$ into groups, and the induced map 
\[
\Ad_*: \Maps(M,SU(2)) \to \Maps (M,SO(3))
\]
into a group homomorphism. The set of spin-structures on $M$ is then identified with the quotient $\Maps(M,SO(3))/ \Im \Ad_*$. The latter set  in general does not carry a natural group structure because $\Im \Ad_*$ need not be a normal subgroup.

Next, fix a reference framing of $M_0$ by restricting the reference framing from $M$. The set of the homotopy classes of almost-framings of $M$ is then identified with the set $[M_0,SO(3)]$ of the homotopy classes of maps $M_0 \to SO(3)$. With respect to this identification and the one from the previous paragraph, the map
\[
\psi: \Maps(M,SO(3))/\Im \Ad_* \longrightarrow [M_0,SO(3)]
\]
is given by the formula $\psi (f) = [f_0]$, where $[f_0]$ stands for the homotopy class of the map $f_0: M_0 \to SO(3)$ obtained by restricting $f: M \to SO(3)$ to $M_0$.

The map $\psi$ is surjective because any map $M_0 \to SO(3)$ extends to a map $M \to SO(3)$ due to the fact that $\pi_2 (SO(3)) = 0$. The map $\psi$ is also injective: suppose $f, g: M \to SO(3)$ are such that their restrictions $f_0, g_0: M_0 \to SO(3)$ are homotopic. Then $g_0 \cdot f_0^{-1}: M_0 \to SO(3)$ is homotopic to the constant map taking the entire $M_0$ to the identity element in $SO(3)$. Therefore, $g_0 \cdot f_0^{-1}$ induces a trivial homomorphism $\pi_1 (M_0) \to \pi_1 (SO(3))$ on the fundamental groups. Keeping in mind that $g_0\,\cdot\,f_0^{-1}$ is the restriction of $g\,\cdot\, f^{-1}$ and that $\pi_1(M) = \pi_1 (M_0)$, we conclude that the map $g\, \cdot\, f^{-1}$ induces a trivial homomorphism $\pi_1 (M) \to \pi_1 (SO(3))$ as well. The lifting criterion applied to the double covering $\Ad: SU(2) \to SO(3)$ then implies that $g \cdot f^{-1}$ lifts to a map $h: M \to SU(2)$ thereby ensuring that $g = \Ad h\cdot f$. This completes the proof.
\end{proof}

\begin{proposition}
\label{7 point 4}
The set of spin-structures on $M$ is in a bijective correspondence with the cohomology group
$H^1 (M;\mathbb Z_2)$.
\end{proposition}

\begin{proof}
Proposition \ref{P:two} identifies the set of spin-structures on $M$ with the set $[M_0,SO(3)]$. Since $SO(3) = \RP^3$, one can use the cellular approximation theorem to identify the latter set with $[M_0,\RP^{\infty}]$. Here, $\RP^{\infty}$ is the infinite dimensional real projective space, which is known to have the homotopy type of the Eilenberg--MacLane space $K(\mathbb Z_2,1)$. Therefore, $[M_0,\RP^{\infty}] = [M_0,K(\mathbb Z_2,1)] = H^1 (M_0;\mathbb Z_2)$; see, for instance Theorem 4.57 of Hatcher \cite{hatcher}. To finish the proof, one simply observes that $H^1 (M_0;\mathbb Z_2) = H^1 (M;\mathbb Z_2)$.
\end{proof}

\begin{remark}
\label{7 point 4 remark}
Proposition \ref{7 point 4} is a well-known fact
and it can be proved in a number of different ways.
See, for instance, \cite[second proposition on page 40]{friedrich}
or \cite{Lawson}.
\end{remark}

\begin{remark}\label{R:milnor}
Milnor \cite{milnor} defines a spin-structure on a CW-complex $M$ as the homotopy class of framings on the 1-skeleton $M^{(1)}$ which can be extended to framings on the 2-skeleton $M^{(2)}$. That this definition is equivalent to the definition of Kaplan \cite{kaplan} for 3-manifolds $M$ follows by combining the fact that $M^{(2)}$ is a deformation retract of $M_0$ with the Puppe exact sequence
\medskip 
\[
0 = \left[\bigvee S^2, SO(3)\right] \longrightarrow \left[M^{(2)},SO(3)\right] \longrightarrow \left[M^{(1)}, SO(3)\right] 
\]

\medskip\noindent
of the cofibration $M^{(1)} \to M^{(2)} \to \bigvee S^2$; see, for instance, Davis--Kirk \cite[Theorem 6.42]{davis-kirk}.
\end{remark}


\end{document}